\documentclass[a4paper,10pt,twoside]{article}
 
\usepackage{amsmath,amssymb,amsthm,dsfont}
\usepackage[english]{babel}
\usepackage[bookmarks]{hyperref} 
\usepackage{mathrsfs} 
\usepackage{color,soul} 

\usepackage{geometry}
\newtheorem{theorem}{Theorem}[section]

\newtheorem{lem}[theorem]{Lemma} 
\newtheorem{prop}[theorem]{Proposition}
\newtheorem{coro}[theorem]{Corollary} 
\theoremstyle{definition}
\newtheorem{rem}[theorem]{Remark}

\newcommand{\N}{\mathbb N}

\newcommand{\R}{\mathbb R}
\newcommand{\C}{\mathbb C}

\newcommand{\Sh}{\mathscr{S}}
\newcommand{\Ph}{\mathscr{P}}
\newcommand{\ep}{\epsilon}
\newcommand{\Gact} {\gamma_{\text{c}}}
\newcommand{\Gace} {\gamma_{\emph{c}}}

\newcommand{\scal}[1]{\left\langle #1 \right\rangle} 
\newcommand{\name}{$\underline{\qquad \qquad}$} 

\newcommand{\defendproof}{\hfill $\Box$} 

\usepackage{fancyhdr} 
\begin{document}
\title{\sc On well-posedness, regularity and ill-posedness for the nonlinear fourth-order Schr\"odinger equation}
\author{\sc{Van Duong Dinh}} 
\date{ }
\maketitle

\begin{abstract}
We prove the local well-posedness for the nonlinear fourth-order Schr\"odinger equation (NL4S) in Sobolev spaces. We also studied the regularity of solutions in the sub-critical case. A direct consequence of this regularity is the global well-posedness above mass and energy spaces under some assumptions. Finally, we show the ill-posedness for (NL4S) in some cases of the super-critical range. 
\end{abstract}


\section{Introduction and main results}
\setcounter{equation}{0}
We consider the Cauchy fourth-order Schr\"odinger equation posed on $\R^d, d\geq 1$, namely
\begin{align}
\left\{
\begin{array}{rl}
i\partial_t u(t,x) + \Delta^2 u(t,x)&=-\mu |u|^{\nu-1} u(t,x), \quad (t, x) \in \R \times \R^d, \\
u(0,x) &= u_0(x), \quad x\in \R^d.
\end{array}
\right.
\tag{NL4S}
\end{align}
where $\nu>1$ and $\mu\in \{\pm 1\}$. The number $\mu=1$ (resp. $\mu=-1$) corresponds to the defocusing case (resp. focusing case). \newline
\indent The fourth-order Schr\"odinger equation was introduced by Karpman \cite{Karpman} and Karpman and Shagalov \cite{KarpmanShagalov} concerning the role of small fourth-order dispersion terms in the propagation of intense laser beams in a bulk medium with Kerr nonlinearity. The study of nonlinear fourth-order Schr\"odinger equation has been attracted a lot of interest in a past decay (see \cite{Pausader}, \cite{Pausadercubic}, \cite{HaoHsiaoWang06}, \cite{HaoHsiaoWang07}, \cite{HuoJia} and references cited therein).  \newline
\indent It is worth noticing that if we set for $\lambda>0$, 
\begin{align}
u_\lambda(t,x)= \lambda^{-\frac{4}{\nu-1}} u( \lambda^{-4} t, \lambda^{-1} x), \label{scaling}
\end{align}
then the (NL4S) is invariant under this scaling. 
An easy computation shows 
\[
\|u_\lambda(0)\|_{\dot{H}^\gamma} = \lambda^{\frac{d}{2}-\frac{4}{\nu-1}- \gamma} \|\varphi\|_{\dot{H}^\gamma}. \nonumber
\]
From this, we define the critical regularity exponent for (NL4S) by
\begin{align}
\Gact =\frac{d}{2} - \frac{4}{\nu-1}. \label{critical exponent 4 schrodinger}
\end{align}
One said that $H^\gamma$ is sub-critical (critical, super-critical) if $\gamma>\Gact$ ($\gamma=\Gact$, $\gamma<\Gact$) respectively. Another important property of (NL4S) is that the following mass and energy are formally conserved under the flow of the equation,
\[
M(u(t))= \int |u(t,x)|^2 dx, \quad E(u(t))= \int \frac{1}{2}|\Delta u(t,x)|^2 + \frac{\mu}{\nu+1}|u(t,x)|^{\nu+1} dx.
\] 
\indent The main purpose of this note is to give the well-posedness and ill-posedness results for (NL4S) in Sobolev spaces. In \cite{Dinh}, the local well-posedness for the nonlinear fractional Schr\"odinger equation including the fourth-order Schr\"odinger equation in both sub-critical and critical cases are showed. We shall review the local well-posedness for the nonlinear fourth-order Schr\"odinger equation below. These results are very similar to the nonlinear Schr\"odinger equation given in \cite{CazenaveWeissler}. We also give the local well-posedness in the critical Sobolev space $H^{d/2}$. The global well-posedness in $L^2$ is then a direct consequence of the local existence and the conservation of mass. We also recall (see e.g. \cite{Pausader} or \cite{Dinh}) the global well-posedness in the energy space $H^2$ under some assumptions. We next show the regularity of solutions in the sub-critical case. As a consequence of this regularity, we obtain the global well-posedness above the mass and energy spaces for (NL4S) under some assumptions. The second part of this note is devoted to the ill-posedness of (NL4S). It is easy to see (e.g \cite{LinSog}) that the (NL4S) is ill-posed in $\dot{H}^\gamma$ for $\gamma<\Gact$. Indeed if $u$ solves the (NL4S) with initial data $\varphi \in \dot{H}^\gamma$ with the lifespan $T$, then the norm $\|u_\lambda(0)\|_{\dot{H}^\gamma}$ and the lifespan of $u_\lambda$ go to zero as $\lambda \rightarrow 0$. Using the technique of Christ-Colliander-Tao given in \cite{ChristCollianderTao}, we are able to prove the ill-posedness for (NL4S) in some cases of the super-critical range, precisely in $H^\gamma$ with $\gamma \in ((-\infty,-d/2] \cap (-\infty, \Gact)) \cup [0,\Gact)$. This ill-posed result is similar to the nonlinear semi-relativistic equation given in \cite{Dinhhalfwave}. Note that for the nonlinear Schr\"odinger equation, the ill-posedness holds in $H^\gamma$ for $\gamma <\max\{0, \Gact\}$ (see \cite{ChristCollianderTao}). The main difference is that the nonlinear Schr\"odinger equation has the Galilean invariance while (NL4S) does not share this property. The Galilean invariance plays a significant role in the proof of the ill-posedness in the range $\gamma\in (-d/2,0)$. Recently, Hong and Sire in \cite{HongSire} used the pseudo-Galilean transformation to get the ill-posedness for the nonlinear fractional Schr\"odinger equation in Sobolev spaces of negative exponent. Unfortunately, it seems to be difficult to control the error of the pseudo-Galilean transformation in high Sobolev norms and so far restricted in one dimension. We finally note that the well-posedness, regularity for the (NL4S) given in this note can be applied for the nonlinear fractional Schr\"odinger equation of order greater than or equal to 2 without any difficulty. Moreover, the ill-posedness argument can be adapted for the nonlinear fractional Schr\"odinger equation of any order. \newline
\indent Before stating our results, let us introduce some notations (see the appendix of \cite{GinibreVelo85}, Chapter 5 of \cite{Triebel} or Chapter 6 of \cite{BerghLosfstom}). Given $\gamma \in \R$ and $1 \leq q \leq \infty$, the generalized Sobolev space is defined by
\[
H^\gamma_q := \Big\{ u \in \Sh' \ | \  \|u\|_{H^\gamma_q}:=\|\scal{\Lambda}^\gamma u\|_{L^q} <\infty \Big\}, \quad \Lambda=\sqrt{-\Delta},
\]
where $\scal{\cdot}$ is the Japanese bracket and $\Sh'$ the space of tempered distributions. The generalized homogeneous Sobolev space is defined by
\[
\dot{H}^\gamma_q := \Big\{ u \in \Sh'_0 \ | \  \|u\|_{\dot{H}^\gamma_q}:=\|\Lambda^\gamma u\|_{L^q} <\infty \Big\},
\]
where $\Sh_0$ is a subspace of the Schwartz space $\Sh$ consisting of functions $\phi$ satisfying $D^\alpha \hat{\phi}(0)=0$ for all $\alpha \in \N^d$ where $\hat{\cdot}$ is the Fourier transform on $\Sh$ and $\Sh'_0$ its topological dual space. One can see $\Sh'_0$ as $\Sh'/\Ph$ where $\Ph$ is the set of all polynomials on $\R^d$. Under these settings, $H^\gamma_q$ and $\dot{H}^\gamma_q$ are Banach spaces with the norms $\|u\|_{H^\gamma_q}$ and $\|u\|_{\dot{H}^\gamma_q}$ respectively. In the sequel, we shall use $H^\gamma:= H^\gamma_2$, $\dot{H}^\gamma:= \dot{H}^\gamma_2$. We also have for $\gamma>0$, $H^\gamma_q = L^q \cap \dot{H}^\gamma_q$. \newline 
\indent Throughout this note, a pair $(p,q)$ is said to be admissible if
\[
(p,q)\in [2,\infty]^2, \quad (p,q,d)\ne (2,\infty,2), \quad \frac{2}{p}+\frac{d}{q}\leq \frac{d}{2}.
\]
We also denote for $(p,q) \in [1,\infty]^2$,
\begin{align}
\gamma_{p,q}=\frac{d}{2}-\frac{d}{q}-\frac{4}{p}. \label{define gamma pq}
\end{align}
Since we are working in spaces of fractional order $\gamma$ or $\beta$, we  need the nonlinearity $F(z)=-\mu |z|^{\nu-1} z$ to have enough regularity. When $\nu$ is an odd integer, $F \in C^\infty(\C, \C)$ (in the real sense). When $\nu$ is not an odd integer, we need the following assumption
\begin{align}
\lceil \gamma \rceil \text{ or } \lceil \beta \rceil \leq \nu, \label{assumption smoothness nonlinearity}
\end{align}
where $\lceil \gamma\rceil$ is the smallest integer greater than or equal to $\gamma$, similarly for $\beta$. Our first result concerns the local well-posedness of (NL4S) in both sub-critical and critical cases.
\begin{theorem} \label{theorem below d/2}
Let $\gamma \in [0,d/2)$ be such that $\gamma \geq \Gace$, and also, if $\nu>1$ is not an odd integer, $(\ref{assumption smoothness nonlinearity})$. 
Let
\begin{align}
p=\frac{8(\nu+1)}{(\nu-1)(d-2\gamma)}, \quad q=\frac{d(\nu+1)}{d+(\nu-1)\gamma}. \label{define pq}
\end{align}
Then for all $u_0 \in H^\gamma$, there exist $T^* \in (0,\infty]$ and a unique solution to \emph{(NL4S)} satisfying
\[
u \in C([0,T^*), H^\gamma) \cap L^p_{\emph{loc}}([0,T^*), H^\gamma_q).
\]
Moreover, the following properties hold:
\begin{itemize}
\item[\emph{(i)}] $u \in L^a_{\emph{loc}}([0,T^*), H^\gamma_b)$ for any admissible pair $(a,b)$ with $b<\infty$ and $\gamma_{a,b}=0$. 
\item[\emph{(ii)}] $M(u(t))=M(u_0)$ for any $t\in [0,T^*)$.
\item[\emph{(iii)}] If $\gamma\geq 2$, $E(u(t))=E(u_0)$ for any $t\in [0,T^*)$.
\item[\emph{(iv)}] If $\gamma >\Gace$ and $T^*<\infty$, then $\|u(t)\|_{\dot{H}^\gamma} \rightarrow \infty$ as $t\rightarrow T^*$.
\item[\emph{(v)}] If $\gamma=\Gace$ and $T^*<\infty$, then $\|u\|_{L^p([0,T^*), H^{\Gace}_q)}=\infty$.
\item[\emph{(vi)}] $u$ depends continuously on $u_0$ in the following sense. There exists $0< T< T^*$ such that if $u_{0,n} \rightarrow u_0$ in $H^\gamma$ and if $u_n$ denotes the solution of \emph{(NL4S)} with initial data $u_{0,n}$, then $0<T< T^*(u_{0,n})$ for all $n$ sufficiently large and $u_n$ is bounded in $L^a([0,T],H^\gamma_b)$ for any admissible pair $(a,b)$ with $\gamma_{a,b}=0$ and $b<\infty$. Moreover, $u_n \rightarrow u$ in $L^a([0,T],L^b)$ as $n \rightarrow \infty$. In particular, $u_n \rightarrow u$ in $C([0,T],H^{\gamma-\ep})$ for all $\ep>0$.
\item[\emph{(vii)}] If $\gamma=\Gace$ and $\|u_0\|_{\dot{H}^{\Gace}}<\varepsilon$ for some $\varepsilon>0$ small enough, then $T^*=\infty$ and the solution is scattering in $H^{\Gace}$, i.e. there exists $u_0^+ \in H^{\Gace}$ such that
\[
\lim_{t\rightarrow+\infty} \|u(t)-e^{it\Delta^2} u_0^+\|_{H^{\Gace}} =0. 
\] 
\end{itemize}
\end{theorem}
We also have the following local well-posedness in the critical Sobolev space $H^{d/2}$.
\begin{theorem} \label{theorem d/2}
Let $\gamma=d/2$ be such that if $\nu>1$ is not an odd integer, $(\ref{assumption smoothness nonlinearity})$. Then for all $u_0 \in H^{d/2}$, there exists $T^*\in (0,\infty]$ and a unique solution to \emph{(NL4S)} satisfying
\[
u \in C([0,T^*),H^{d/2}) \cap L^p_{\emph{loc}}([0,T^*),L^\infty),
\]
for some $p>\max(\nu-1,4)$ when $d=1$ and some $p>\max(\nu-1,2)$ when $d\geq 2$. Moreover, the following properties hold:
\begin{itemize}
\item[\emph{(i)}]  $u \in L^a_{\emph{loc}}([0,T^*),H^{d/2}_b)$ for any admissible pair $(a,b)$ with $b<\infty$ and $\gamma_{a,b}=0$.
\item[\emph{(ii)}] If $T^*<\infty$, then $\|u(t)\|_{H^{d/2}} \rightarrow \infty$ as $t \rightarrow T^*$.
\item[\emph{(iii)}] $u$ depends continuously on $u_0$ in the sense of \emph{Theorem $\ref{theorem below d/2}$}
\end{itemize}
\end{theorem}
The continuous dependence can be improved (see Remark $\ref{rem continuity d/2}$) if we assume that $\nu>1$ is an odd integer or $\lceil d/2 \rceil \leq \nu-1$. Concerning the well-posedness of the nonlinear Schr\"odinger equation in this critical space, we refer to \cite{Kato95} and \cite{NakamitsuOzawa}. Note that in \cite{NakamitsuOzawa}, the global well-posedness with small data is proved with exponential-type nonlinearity but not the local well-posedness without size restriction on the initial data. \newline  
\indent It is well-known that (see Chapter 4 of \cite{Cazenave}, \cite{Kato95} or Chapter 3 of \cite{Tao}) that for $\gamma>d/2$, the nonlinear Schr\"odinger equation is locally well-posed provided the nonlinearity has enough regularity. It is not a problem to extend this result for the nonlinear fourth-order Schr\"odinger equation. For the sake of completeness, we state (without proof) the local well-posedness for (NL4S) in this range. 
\begin{theorem} \label{theorem above d/2}
Let $\gamma>d/2$ be such that if $\nu>1$ is not an odd integer, $(\ref{assumption smoothness nonlinearity})$. Then for all $u_0 \in H^\gamma$, there exist $T^*\in (0,\infty]$ and a unique solution $u \in C([0,T^*),H^\gamma)$ to \emph{(NL4S)}. Moreover, the following properties hold:
\begin{itemize}
\item[\emph{(i)}] $u \in L^a_{\emph{loc}}([0,T^*),H^\gamma_b)$ for any admissible pair $(a,b)$ with $b<\infty$ and $\gamma_{a,b}=0$.
\item[\emph{(ii)}] If $T^*<\infty$, then $\|u(t)\|_{H^\gamma} \rightarrow \infty$ and $\limsup \|u(t)\|_{L^\infty} \rightarrow \infty$ as $t\rightarrow T^*$.
\item[\emph{(iii)}] $u$ depends continuously on $u_0$ in the following sense. There exists $0<T<T^*$ such that if $u_{0,n}\rightarrow u_0$ in $H^\gamma$ and if $u_n$ is the solution of \emph{(NL4S)} with the initial data $u_{0,n}$, then $u_n\rightarrow u$ in $C([0,T], H^\gamma)$.  
\end{itemize}
\end{theorem} 
\begin{coro}\label{coro global L2}
Let $\nu \in (1, 1+ 8/d)$. Then for all $\varphi \in L^2$, there exists a unique global solution to \emph{(NL4S)} satisfying $u \in C(\R, L^2) \cap L^p_{\emph{loc}}(\R, L^q)$, where $(p,q)$ given in $(\ref{define pq})$.
\end{coro}
In the energy space $H^2$, we have the following global well-posedness result.
\begin{prop}[\cite{Pausader} or \cite{Dinh}] \label{prop global H2}
Let $\nu\in (1,1+8/(d-4))$ for $d \geq 5$ and $\nu>1$ for $d\leq 4$. Then for any $u_0\in H^2$, the solution to \emph{(NL4S)} given in \emph{Theorem \ref{theorem below d/2}, Theorem $\ref{theorem d/2}$ and Theorem $\ref{theorem above d/2}$} can be extended to the whole $\R$ if one of the following is satisfied:
\begin{itemize}
\item[\emph{(i)}] $\mu=1$.
\item[\emph{(ii)}] $\mu =-1, \nu <1+8/d$.
\item[\emph{(iii)}] $\mu =-1, \nu=1+8/d$ and $\|u_0\|_{L^2}$ is small.
\item[\emph{(iv)}] $\mu=-1$ and $\|u_0\|_{H^{2}}$ is small.
\end{itemize}
\end{prop}
Our next result concerns with the regularity of solutions of (NL4S) in the sub-critical case.
\begin{theorem} \label{theorem regularity}
Let $\beta>\gamma \geq 0$ be such that $\gamma> \Gace$, and also, if $\nu>1$ is not an odd integer, $(\ref{assumption smoothness nonlinearity})$. Let $u_0 \in H^\gamma$ and $u$ be the corresponding $H^\gamma$ solution of \emph{(NL4S)} given in \emph{Theorem $\ref{theorem below d/2}$, Theorem $\ref{theorem d/2}$, Theorem $\ref{theorem above d/2}$}. If $u_0 \in H^{\beta}$, then $u \in C([0,T^*),H^\beta)$. 
\end{theorem}
The following result is a direct consequence of Theorem $\ref{theorem regularity}$ and the global well-posedness in Corollary $\ref{coro global L2}$ and Proposition $\ref{prop global H2}$.
\begin{coro} \label{coro global}
\begin{itemize}
\item[\emph{(i)}] Let $\gamma \geq 0$ and $\nu \in (1,1+8/d)$ be such that if $\nu$ is not an odd integer, $(\ref{assumption smoothness nonlinearity})$. Then \emph{(NL4S)} is globally well-posed in $H^\gamma$. 
\item[\emph{(ii)}] Let $\gamma \geq 2$, $\nu \in [1+8/d,1+8/(d-4))$ for $d \geq 5$ and $\nu \in [1+8/d,\infty)$ for $d\leq 4$ be such that if $\nu$ is not an odd integer, $(\ref{assumption smoothness nonlinearity})$. Then \emph{(NL4S)} is globally well-posed in $H^\gamma$ provided one of conditions \emph{(i), (iii), (iv)} in \emph{Proposition $\ref{prop global H2}$} is satisfied. 
\end{itemize}
\end{coro}
Our final result is the following ill-posedness for (NL4S).
\begin{theorem} \label{theorem ill-posedness}
Let $\nu>1$ be such that if $\nu$ is not an odd integer, $\nu\geq k+1$ for some integer $k>d/2$. Then \emph{(NL4S)} is ill-posed in $H^\gamma$ for $\gamma \in ((-\infty,-d/2]\cap (-\infty,\Gace)) \cup [0,\Gace)$. Precisely, if $\gamma \in ((-\infty,-d/2]\cap (-\infty,\Gace)) \cup (0,\Gace)$, then for any $t>0$ the solution map $\Sh \ni u(0)\mapsto u(t)$ of \emph{(NL4S)} fails to be continuous at 0 in the $H^\gamma$ topology. Moreover, if $\Gace>0$, the solution map fails to be uniformly continuous on $L^2$.
\end{theorem}
The proof of Theorem $\ref{theorem ill-posedness}$ bases on the small dispersion analysis given in \cite{ChristCollianderTao}. Note that when $\nu=3$ and $\mu=1$ corresponding to the defocusing cubic nonlinearity, Pausader in \cite{Pausadercubic} proves the ill-posedness for (NL4S) in $H^2(\R^d)$ with $d\geq 9$. \newline 
\indent This note is organized as follows. In Section $\ref{section well-posedness}$, we recall Strichartz estimate for the linear fourth-order Schr\"odinger equation and the nonlinear fractional derivatives. We end this section with the proof the local well-posedness given in Theorem $\ref{theorem below d/2}$ and Theorem $\ref{theorem d/2}$. In Section $\ref{section regularity}$, we give the proof of the regularity for solutions of (NL4S) given in Theorem $\ref{theorem regularity}$. Finally, the proof of the ill-posedness result is given in Section $\ref{section ill-posedness}$.
\section{Well-posedness} \label{section well-posedness}
\setcounter{equation}{0}
In this section, we will give the proofs of the local well-posedness given in Theorem $\ref{theorem below d/2}$ and Theorem $\ref{theorem d/2}$. Our proofs are based on the standard contraction mapping argument using Strichartz estimate and nonlinear fractional derivatives (see Subsection $\ref{subsection fractional derivative}$). 
\subsection{Strichartz estimate}
In this subsection, we recall Strichartz estimate for the fourth-order Schr\"odinger equation.
\begin{prop}[\cite{Dinh}] \label{prop strichartz}
Let $\gamma \in \R$ and $u$ be a (weak) solution to the linear fourth-order Schr\"odinger equation, namely
\[
u(t)=e^{it\Delta^2}u_0+\int_0^t e^{i(t-s)\Delta^2} F(s)ds,
\]
for some data $u_0, F$. Then for all $(p,q)$ and $(a,b)$ admissible with $q<\infty$ and $b<\infty$,
\begin{align}
\|u\|_{L^p(\R, L^q)} \lesssim \|u_0\|_{\dot{H}^{\gamma_{p,q}}} +\|F\|_{L^{a'}(\R, L^{b'})}, \label{strichartz estimate}
\end{align}
provided that
\begin{align}
\gamma_{p,q}=\gamma_{a',b'}+4.
\end{align}
Here $(a,a')$ is a conjugate pair and similarly for $(b,b')$.
\end{prop} 
\begin{rem} \label{rem strichartz}
The estimate $(\ref{strichartz estimate})$ is exactly the one given in \cite{Pausader} or \cite{Pausadercubic} where the author considered $(p,q)$ and $(a,b)$ are Schr\"odinger admissible, i.e.
\[
p, q \in [2,\infty]^2, \quad (p,q,d) \ne (2,\infty,2), \quad \frac{2}{p}+\frac{d}{q}=\frac{d}{2}.
\]
We refer to \cite{Dinh} for the proof of Proposition $\ref{prop strichartz}$. Note that rather than using directly a dedicate dispersive estimate of \cite{Ben-ArtziKochSaut} for the fundamental solution of the homogeneous fourth-order Schr\"odinger equation, we use scaling technique which is similar to those of wave equation (see e.g. \cite{KeelTaoTTstar}). Our Strichartz estimate is flexible enough to show the local well-posedness for (NL4S) in both sub-critical and critical cases. 
\end{rem}
We also have the following local Strichartz estimate (see again \cite{Dinh}).
\begin{coro} \label{coro local strichartz}
Let $\gamma \geq 0$ and $I$ be a bounded interval. If $u$ is a weak solution to the linear fourth-order Schr\"odinger equation for some data $u_0, F$, then for all $(p,q)$ admissible satisfying $q<\infty$,
\begin{align}
\|u\|_{L^p(I,H^{\gamma-\gamma_{p,q}}_q)} \lesssim \|u_0\|_{H^\gamma} + \|F\|_{L^1(I,H^\gamma)}. \label{local strichartz}
\end{align}
\end{coro}
\subsection{Nonlinear fractional derivatives} \label{subsection fractional derivative}
In this subsection, we recall some nonlinear fractional derivatives estimates related to our purpose. Let us start with the following fractional Leibniz rule (or Kato-Ponce inequality). We refer to \cite{GrafakosOh} for the proof of a more general result.
\begin{prop} \label{prop fractional leibniz}
Let $\gamma \geq 0, 1<r<\infty$ and $1<p_1, p_2, q_1, q_2 \leq \infty$ satisfying
\[
\frac{1}{r}=\frac{1}{p_1}+\frac{1}{q_1}=\frac{1}{p_2}+\frac{1}{q_2}.
\]
Then there exists $C=C(d,\gamma, r, p_1, q_1, p_2, q_2)>0$  such that for all $u, v \in \Sh$,
\[
\|\Lambda^\gamma(uv)\|_{L^r} \leq C\Big( \|\Lambda^\gamma u\|_{L^{p_1}} \|v\|_{L^{q_1}} + \|u\|_{L^{p_2}} \|\Lambda^\gamma v\|_{L^{q_2}}\Big).
\]
\end{prop}
We also have the following fractional chain rule (see \cite{ChristWeinstein} or \cite{Staffilani}).
\begin{prop}\label{prop fractional chain}
Let $F \in C^1(\C, \C)$ and $G \in C(\C, \R^+)$ such that $F(0)=0$ and 
\[
|F'( \theta z+ (1-\theta) \zeta)| \leq \mu(\theta) (G(z)+G(\zeta)), \quad z,\zeta \in \C, \quad 0 \leq \theta \leq 1,
\]
where $\mu \in L^1((0,1))$. Then for $\gamma \in (0,1)$ and $1 <r, p <\infty$, $1 <q \leq \infty$ satisfying 
\[
\frac{1}{r}=\frac{1}{p}+\frac{1}{q},
\] 
there exists $C=C(d,\mu, \gamma, r, p,q)>0$ such that for all $u \in \Sh$,
\[
\|\Lambda^\gamma F(u) \|_{L^r} \leq C \|F'(u) \|_{L^q} \|\Lambda^\gamma u \|_{L^p}.
\]
\end{prop}
Combining the fractional Leibniz rule and the fractional chain rule, one has the following result (see the appendix of \cite{Kato95}).
\begin{lem} \label{lem nonlinear estimates}
Let $F \in C^k(\C, \C), k \in \N \backslash \{0\}$. Assume that there is $\nu \geq k$  such that 
\[
|D^iF (z)| \leq C |z|^{\nu-i}, \quad z \in \C, \quad i=1,2,...., k.
\]
Then for $\gamma \in [0,k]$ and $1 <r, p <\infty$, $1 <q \leq \infty$ satisfying $\frac{1}{r}=\frac{1}{p}+\frac{\nu-1}{q}$, there exists $C=C(d, \nu, \gamma, r,p,q)>0$ such that for all $u \in \Sh$,
\begin{align}
\|\Lambda^\gamma F(u)\|_{L^r} \leq C \|u\|^{\nu-1}_{L^q} \|\Lambda^\gamma u \|_{L^p}. \label{nonlinear estimate} 
\end{align}
Moreover, if $F$ is a homogeneous polynomial in $u$ and $\overline{u}$, then $(\ref{nonlinear estimate})$ holds true for any $\gamma \geq 0$.
\end{lem}
\subsection{Proof of Theorem $\ref{theorem below d/2}$}
We are now able to prove Theorem $\ref{theorem below d/2}$. Let us firstly comment about the choice of $(p,q)$ given in $(\ref{define pq})$. It is easy to see that $(p,q)$ is admissible and $\gamma_{p,q}=0=\gamma_{p',q'}+4$. This allows us to use Strichartz estimate $(\ref{strichartz estimate})$ for $(p,q)$. Moreover, if we choose $(m,n)$ so that
\begin{align}
\frac{1}{p'}=\frac{1}{m}+\frac{\nu-1}{p}, \quad \frac{1}{q'}=\frac{1}{q}+\frac{\nu-1}{n}, \label{define mn}
\end{align}
Thanks to this choice of $n$, we have the Sobolev embedding $\dot{H}^\gamma_q \hookrightarrow L^n$ since
\[
q \leq n=\frac{dq}{d-\gamma q}.
\]
\textbf{Step 1.}  Existence. Let us consider 
\[
X:=\Big\{u \in L^p(I, H^\gamma_q) \ | \ \|u\|_{L^p(I,\dot{H}^\gamma_q)} \leq M
\Big\},
\]
equipped with the distance 
\[
d(u,v)=\|u-v\|_{L^p(I,L^q)},
\]
where $I=[0,T]$ and $M, T>0$ to be chosen later. It is easy to verify (see e.g. \cite{CazenaveWeissler} or Chapter 4 of  \cite{Cazenave}) that $(X,d)$ is a complete metric space. By the Duhamel formula, it suffices to prove that the functional 
\begin{align}
\Phi(u)(t)= e^{it\Delta^2} u_0 + i\mu \int_0^t e^{i(t-s)\Delta^2}|u(s)|^{\nu-1} u(s) ds=: u_{\text{hom}}(t)+ u_{\text{inh}}(t) \label{duhamel formula}
\end{align}
is a contraction on $(X,d)$. \newline
\indent Let us firstly consider the case $\gamma>\Gact$. In this case, we have $1<m<p$ and
\begin{align}
\frac{1}{m}-\frac{1}{p}=1-\frac{(\nu-1)(d-2\gamma)}{8}=:\theta>0. \label{subcritical}
\end{align}
Using Strichartz estimate $(\ref{strichartz estimate})$, we obtain
\begin{align*}
\|\Phi(u)\|_{L^p(I, \dot{H}^\gamma_q)} &\lesssim \|u_0\|_{\dot{H}^\gamma}+ \|F(u)\|_{L^{p'}(I,\dot{H}^\gamma_{q'})}, \\
\|\Phi(u)-\Phi(v)\|_{L^p(I,L^q)} &\lesssim \|F(u)-F(v)\|_{L^{p'}(I,L^{q'})},
\end{align*}
where $F(u)=|u|^{\nu-1}u$ and similarly for $F(v)$. It then follows from Lemma $\ref{lem nonlinear estimates}$, $(\ref{define mn})$, Sobolev embedding and $(\ref{subcritical})$ that
\begin{align}
\|F(u)\|_{L^{p'}(I,\dot{H}^\gamma_{q'})} &\lesssim T^\theta\|u\|^\nu_{L^p(I,\dot{H}^\gamma_q)}, \label{blowup subcritical}\\
\|F(u)-F(v)\|_{L^{p'}(I,L^{q'})} &\lesssim T^\theta \Big(\|u\|^{\nu-1}_{L^p(I,\dot{H}^\gamma_q)} + \|v\|^{\nu-1}_{L^p(I,\dot{H}^\gamma_q)}\Big)\|u-v\|_{L^p(I,L^q)}. \label{uniqueness subcritical}
\end{align}
This shows that for all $u, v \in X$, there exists $C>0$ independent of $T$ and $u_0 \in H^\gamma$ such that
\begin{align*}
\|\Phi(u)\|_{L^p(I,\dot{H}^\gamma_q)} &\leq C\|u_0\|_{\dot{H}^\gamma} +C T^\theta M^\nu, \\
d(\Phi(u),\Phi(v)) &\leq CT^\theta M^{\nu-1} d(u,v).
\end{align*} 
If we set $M=2C\|u_0\|_{\dot{H}^\gamma}$ and choose $T>0$ so that
\[
CT^\theta M^{\nu-1} \leq \frac{1}{2},
\]
then $\Phi$ is a strict contraction on $(X,d)$. \newline
\indent We now turn to the case $\gamma=\Gact$. We have from Strichartz estimate $(\ref{strichartz estimate})$ that
\[
\|u_{\text{hom}}\|_{L^p(I,\dot{H}^{\Gact}_q)} \lesssim \|u_0\|_{\dot{H}^{\Gact}}.
\]
This shows that $\|u_{\text{hom}}\|_{L^p(I,\dot{H}^{\Gact}_q)} \leq \varepsilon$ for some $\varepsilon>0$ small enough provided that $T$ is small or $\|u_0\|_{\dot{H}^{\Gact}}$ is small. We also have from $(\ref{strichartz estimate})$ that
\[
\|u_{\text{inh}}\|_{L^p(I,\dot{H}^{\Gact}_q)} \lesssim \|F(u)\|_{L^{p'}(I,\dot{H}^{\Gact}_{q'})}. 
\]
Lemma $(\ref{lem nonlinear estimates})$, $(\ref{define mn})$ and Sobolev embedding (note that in this case $m=p$) then yield that 
\begin{align}
\|F(u)\|_{L^{p'}(I,\dot{H}^{\Gact}_{q'})} &\lesssim \|u\|^\nu_{L^p(I,\dot{H}^{\Gact}_q)}, \label{blowup critical} \\
\|F(u)-F(v)\|_{L^{p'}(I,L^{q'})} &\lesssim \Big(\|u\|^{\nu-1}_{L^p(I,\dot{H}^{\Gact}_q)} + \|v\|^{\nu-1}_{L^p(I,\dot{H}^{\Gact}_q)}\Big) \|u-v\|_{L^p(I,L^q)}. \label{uniqueness critical}
\end{align}
This implies that for all $u, v \in X$, there exists $C>0$ independent of $T$ and $u_0 \in H^{\Gact}$ such that
\begin{align*}
\|\Phi(u)\|_{L^p(I,\dot{H}^{\Gact}_q)} &\leq \varepsilon +C M^\nu, \\
d(\Phi(u),\Phi(v)) &\leq CM^{\nu-1}d(u,v).
\end{align*}
If we choose $\varepsilon$ and $M$ small so that
\[
C M^{\nu-1} \leq \frac{1}{2}, \quad \varepsilon +\frac{M}{2}\leq M,
\]
then $\Phi$ is a contraction on $(X,d)$. \newline
\indent Therefore, in both sub-critical and critical cases, $\Phi$ has a unique fixed point in $X$. Moreover, since $u_0 \in H^\gamma$ and $u \in L^p(I,H^\gamma_q)$, the Strichartz estimate shows that $u \in C(I,H^\gamma)$ (see e.g. \cite{CazenaveWeissler} or Chapter 4 of \cite{Cazenave}). This shows the existence of solution $u \in C(I,H^\gamma)\cap L^p(I,H^\gamma_q)$ to (NL4S). Note that in the case $\gamma=\Gact$, if $\|u_0\|_{\dot{H}^{\Gact}}$ is small enough, then we can take $T=\infty$.\newline
\textbf{Step 2.} Uniqueness. It follows easily from $(\ref{uniqueness subcritical})$ and $(\ref{uniqueness critical})$ using the fact that $\|u\|_{L^p(I,\dot{H}^\gamma_q)}$ can be small if $T$ is small. \newline
\textbf{Step 3.} Item (i). Let $u \in C(I,H^\gamma) \cap L^p(I,H^\gamma_q)$ be a solution to (NLFS) where $I=[0,T]$ and $(a,b)$ an admissible pair with $b<\infty$ and $\gamma_{a,b}=0$. Then Strichartz estimate $(\ref{strichartz estimate})$ implies
\begin{align}
\|u\|_{L^a(I,L^b)} &\lesssim \|u_0\|_{L^2} + \|F(u)\|_{L^{p'}(I,L^{q'})}, \label{estimate 1} \\
\|u\|_{L^a(I,\dot{H}^\gamma_b)} &\lesssim \|u_0\|_{\dot{H}^\gamma}+ \|F(u)\|_{L^{p'}(I,\dot{H}^\gamma_{q'})}. \label{estimate 2}
\end{align}
It then follows from $(\ref{blowup subcritical})$ and $(\ref{blowup critical})$ that $u \in L^a(I,H^\gamma_b)$. \newline
\textbf{Step 4.} Item (ii) and (iii). The conservation of mass and energy follows similarly as for the Schr\"odinger equation (see e.g. \cite{CazenaveWeissler}, Chapter 4 of \cite{Cazenave} or Chapter 5 of \cite{Ginibre}). \newline
\textbf{Step 5.} Item (iv). The blowup alternative in sub-critical case is easy since the time of existence depends only on $\|u_0\|_{\dot{H}^\gamma}$. \newline
\textbf{Step 6.} Item (v). It also follows from a standard argument (see e.g. \cite{CazenaveWeissler}). Indeed, if $T^*<\infty$ and $\|u\|_{L^p([0,T^*), H^{\Gact}_q)} <\infty$, then Strichartz estimate $(\ref{strichartz estimate})$ implies that $u \in C([0,T^*],H^{\Gact})$. Thus, one can extend the solution to (NL4S) beyond $T^*$. It leads to a contradiction with the maximality of $T^*$. \newline
\textbf{Step 7.} Item (vi). We use the argument given in \cite{CazenaveWeissler}. From Step 1, in the sub-critical case, we can choose $T$ and $M$ so that the fixed point argument can be carried out on $X$ for any initial data with $\dot{H}^\gamma$ norm less than $2\|u_0\|_{\dot{H}^\gamma}$. In the critical case, there exist $T, M$ and an $\dot{H}^{\Gact}$ neighborhood $U$ of $u_0$ such that the fixed point argument can be carried out on $X$ for all initial data in $U$. Now let $u_{0,n}\rightarrow u_0$ in $H^\gamma$. In both sub-critical and critical cases, we see that $T<T^*(u_0)$, $\|u\|_{L^p([0,T], \dot{H}^\gamma_q)} \leq M$, and that for sufficiently large $n$, $T<T^*(u_{0,n})$ and $\|u_n\|_{L^p([0,T],\dot{H}^\gamma_q)} \leq M$. Thus, $(\ref{estimate 1})$ and $(\ref{estimate 2})$ together with $(\ref{blowup subcritical})$ and $(\ref{blowup critical})$ yield that $u_n$ is bounded in $L^a([0,T],H^\gamma_b)$ for any admissible pair $(a,b)$ with $b<\infty$ and $\gamma_{a,b}=0$. We also have from $(\ref{uniqueness subcritical})$, $(\ref{uniqueness critical})$ and the choice of $T$ that
\[
d(u_n,u) \leq C\|u_{0,n}-u_0\|_{L^2}+\frac{1}{2} d(u_n,u) \text{ or } d(u_n, u) \leq 2C\|u_{0,n}-u_0\|_{L^2}. 
\]
This shows that $u_n \rightarrow u$ in $L^p([0,T], L^q)$. Again $(\ref{estimate 2})$ together with $(\ref{uniqueness subcritical})$ and $(\ref{uniqueness critical})$ implies that $u_n \rightarrow u$ in $L^a([0,T],L^b)$ for any admissible pair $(a,b)$ with $b<\infty$ and $\gamma_{a,b}=0$. The convergence in $C(I,H^{\gamma-\ep})$ follows from the boundedness in $L^\infty(I,H^\gamma)$ and the convergence in $L^\infty(I,L^2)$ and that $\|u\|_{H^{\gamma-\ep}} \leq \|u\|^{1-\frac{\ep}{\gamma}}_{H^\gamma} \|u\|^{\frac{\ep}{\gamma}}_{L^2}$. \newline
\textbf{Step 8.} Item (vii). As mentioned in Step 1, when $\|u_0\|_{\dot{H}^{\Gact}}$ is small, we can take $T^*=\infty$. It remains to prove the scattering property. To do so, we make use of the adjoint estimate to the homogeneous Strichartz estimate, namely $L^2 \ni u_0 \mapsto e^{it\Delta^2} u_0 \in L^p(\R, L^q)$ to obtain
\begin{align}
\|e^{-it_2\Delta^2}u(t_2)-e^{-it_1\Delta^2} u(t_1)\|_{\dot{H}^{\Gact}} &= \Big\|i \mu \int_{t_1}^{t_2} e^{-is\Delta^2} F(u)(s) ds \Big\|_{\dot{H}^{\Gact}} \nonumber\\
&= \Big\|i \mu \int_{t_1}^{t_2} \Lambda^{\Gact}e^{-is\Delta^2} (\mathds{1}_{[t_1,t_2]}F(u))(s) ds \Big\|_{L^2} \nonumber \\
&\lesssim \|F(u)\|_{L^{p'}([t_1,t_2],\dot{H}^{\Gact}_{q'})}. \label{scattering 1}
\end{align}
Similarly,
\begin{align}
\|e^{-it_2\Delta^2}u(t_2)-e^{-it_1\Delta^2} u(t_1) \|_{L^2} \lesssim \|F(u)\|_{L^{p'}([t_1,t_2],L^{q'})}. \label{scattering 2}
\end{align}
Thanks to $(\ref{blowup critical})$ and $(\ref{uniqueness critical})$, we get
\[
\|e^{-it_2\Delta^2}u(t_2)-e^{-it_1\Delta^2} u(t_1)\|_{H^{\Gact}} \rightarrow 0,
\]
as $t_1, t_2 \rightarrow +\infty$. This implies that the limit
\[
u^+_0:=\lim_{t\rightarrow +\infty} e^{-it\Delta^2} u(t)
\]
exists in $H^{\Gact}$. Moreover, 
\[
u(t)-e^{it\Delta^2} u^+_0 = -i\mu \int_t^{+\infty} e^{i(t-s)\Delta^2} F(u(s)) ds.
\]
Using again $(\ref{scattering 1})$ and $(\ref{scattering 2})$ together with $(\ref{blowup critical})$ and $(\ref{uniqueness critical})$, we have
\[
\lim_{t\rightarrow +\infty} \|u(t)-e^{it\Delta^2} u^+_0\|_{H^{\Gact}}=0. 
\] 
This completes the proof of Theorem $\ref{theorem below d/2}$. \defendproof
\subsection{Proof of Theorem $\ref{theorem d/2}$}
We now turn to the proof of the local well-posedness in $H^{d/2}$. To do so, we firstly choose $p>\max(\nu-1,4)$ when $d=1$ and $p>\max(\nu-1,2)$ when $d\geq 2$ and then choose $q \in [2,\infty)$ such that 
\[
\frac{2}{p} +\frac{d}{q}\leq \frac{d}{2}.
\]
\textbf{Step 1.} Existence. We will show that $\Phi$ defined in $(\ref{duhamel formula})$ is a contraction on
\[
X:= \Big\{ u \in L^\infty(I,H^{d/2}) \cap L^p(I, H^{d/2-\gamma_{p,q}}_q) \  | \ \|u\|_{L^\infty(I,H^{d/2})}+\|u\|_{L^p(I,H^{d/2-\gamma_{p,q}}_q)} \leq M \Big\},
\]
equipped with the distance
\[
d(u,v):= \|u-v\|_{L^\infty(I,L^2)} +\|u-v\|_{L^p(I,H^{-\gamma_{p,q}})},
\]
where $I=[0,T]$ and $M,T>0$ to be determined.  The local Strichartz estimate $(\ref{local strichartz})$ gives
\begin{align*}
\|\Phi(u)\|_{L^\infty(I,H^{d/2})} +\|\Phi(u)\|_{L^p(I,H^{d/2-\gamma_{p,q}}_q)} &\lesssim \|u_0\|_{H^{d/2}} + \|F(u)\|_{L^1(I,H^{d/2})},\\
\|\Phi(u)-\Phi(v)\|_{L^\infty(I,L^2)} + \|\Phi(u)-\Phi(v)\|_{L^p(I,H^{-\gamma_{p,q}}_q)} &\lesssim \|F(u)-F(v)\|_{L^1(I,L^2)}.
\end{align*}
Thanks to the assumptions on $\nu$, Lemma $\ref{lem nonlinear estimates}$ implies
\begin{align}
\|F(u)\|_{L^1(I,H^{d/2})} &\lesssim \|u\|^{\nu-1}_{L^{\nu-1}(I,L^\infty)} \|u\|_{L^\infty(I,H^{d/2})} \lesssim T^\theta \|u\|^{\nu-1}_{L^p(I,L^\infty)} \|u\|_{L^\infty(I,H^{d/2})}, \label{blowup d/2} \\
\|F(u)-F(v)\|_{L^1(I,L^2)} &\lesssim \Big(\|u\|^{\nu-1}_{L^{\nu-1}(I,L^\infty)} + \|v\|^{\nu-1}_{L^{\nu-1}(I,L^\infty)} \Big) \|u-v\|_{L^\infty(I,L^2)}  \nonumber \\
&\lesssim T^\theta \Big(\|u\|^{\nu-1}_{L^p(I,L^\infty)} + \|v\|^{\nu-1}_{L^p(I,L^\infty)} \Big) \|u-v\|_{L^\infty(I,L^2)}, \label{uniqueness d/2}
\end{align}
where $\theta=1-\frac{\nu-1}{p}>0$. Using the fact that $d/2-\gamma_{p,q}>d/q$, the Sobolev embedding implies $H^{d/2-\gamma_{p,q}}_q \hookrightarrow L^\infty$. Thus,
\begin{align*}
\|\Phi(u)\|_{L^\infty(I,H^{d/2})} +\|\Phi(v)\|_{L^p(I,H^{d/2-\gamma_{p,q}}_q)} &\lesssim \|u_0\|_{H^{d/2}} + T^\theta \|u\|^{\nu-1}_{L^p(I,H^{d/2-\gamma_{p,q}}_q)} \|u\|_{L^\infty(I,H^{d/2})}, \\
d(\Phi(u),\Phi(v)) &\lesssim  T^\theta \Big(\|u\|^{\nu-1}_{L^p(I,H^{d/2-\gamma_{p,q}}_q)} + \|v\|^{\nu-1}_{L^p(I,H^{d/2-\gamma_{p,q}}_q)}\Big) d(u,v).
\end{align*}
Thus for all $u, v \in X$, there exists $C>0$ independent of $u_0 \in H^{d/2}$ such that
\begin{align*}
\|\Phi(u)\|_{L^\infty(I,H^{d/2})} +\|\Phi(v)\|_{L^p(I,H^{d/2-\gamma_{p,q}}_q)} &\leq C \|u_0\|_{H^{d/2}} +CT^\theta M^\nu, \\
d(\Phi(u),\Phi(v)) &\leq CT^\theta M^{\nu-1} d(u,v).
\end{align*}
If we set $M=2C\|u_0\|_{H^{d/2}}$ and choose $T>0$ small enough so that $CT^\theta M^{\nu-1} \leq \frac{1}{2}$, then $\Phi$ is a contraction on $X$. \newline
\textbf{Step 2.} Uniqueness. It is easy using $(\ref{uniqueness d/2})$ since $\|u\|_{L^p(I,L^\infty)}$ is small if $T$ is small. \newline
\textbf{Step 3.} Item (i). It follows easily from Step 1 and Strichartz estimate $(\ref{local strichartz})$ that for any admissible pair $(a,b)$ with $b<\infty$ and $\gamma_{a,b}=0$,
\[
\|u\|_{L^a(I,H^{d/2}_b)} \lesssim \|u_0\|_{H^{d/2}}+ \|F(u)\|_{L^1(I,H^{d/2})}.
\] 
\textbf{Step 4.} Item (ii). The blowup alternative is obvious since the time of existence depends only on $\|u_0\|_{H^{d/2}}$. \newline
\textbf{Step 5.} Item (iii). The continuous dependence is similar to Step 7 of the proof of Theorem $\ref{theorem below d/2}$ using $(\ref{uniqueness d/2})$.
\defendproof
\begin{rem} \label{rem continuity d/2}
If we assume that $\nu>1$ is an odd integer or $\lceil d/2 \rceil \leq \nu-1$ otherwise, then the continuous dependence holds in $C(I,H^{d/2})$. Indeed, we consider $X$ as above equipped with the following metric
\[
d(u,v):= \|u-v\|_{L^\infty(I,H^{d/2})} + \|u-v\|_{L^p(I,H^{d/2-\gamma_{p,q}}_q)}.
\]
Thanks to the assumptions on $\nu$, we are able to apply the fractional derivatives  estimates (see e.g. the appendix of \cite{Kato95} or Corollary 3.5 of \cite{Dinh}) to have
\begin{multline*}
\|F(u)-F(v)\|_{L^1(I,H^{d/2})} \lesssim (\|u\|^{\nu-1}_{L^{\nu-1}(I,L^\infty)} 
+ \|v\|^{\nu-1}_{L^{\nu-1}(I,L^\infty)}) \|u-v\|_{L^\infty(I,H^{d/2})} \\
+ (\|u\|^{\nu-2}_{L^{\nu-1}(I,L^\infty)} + \|v\|^{\nu-2}_{L^{\nu-1}(I,L^\infty)})(\|u\|_{L^\infty(I,H^{d/2})} +\|v\|_{L^\infty(I,H^{d/2})}) \|u-v\|_{L^{\nu-1}(I,L^\infty)}.
\end{multline*}
The Sobolev embedding then implies that for all $u,v \in X$,
\[
d(\Phi(u),\Phi(v)) \lesssim T^\theta M^{\nu-1} d(u,v).
\]
The continuous dependence in $C(I,H^{d/2})$ follows as Step 7 of the proof of Theorem $\ref{theorem below d/2}$.
\end{rem}
\section{Regularity} \label{section regularity}
\setcounter{equation}{0}
The main purpose of this section is to prove the regularity of solutions of (NL4S) given in Theorem $\ref{theorem regularity}$. We follow the argument given in Chapter 5 of \cite{Cazenave}. To do so, we will split $\gamma$ into three cases $\gamma \in [0,d/2)$, $\gamma=d/2$ and $\gamma>d/2$. 
\subsection{The case $\gamma\in [0,d/2)$}
Let $\beta >\gamma$. If $u_0 \in H^\beta$, then Theorem $\ref{theorem below d/2}$ or Theorem $\ref{theorem d/2}$ or Theorem $\ref{theorem above d/2}$ shows that there exists a maximal solution to (NL4S) satisfying $u \in C([0,T),H^\beta) \cap L^a_{\text{loc}}([0,T),H^\beta_b)$ for any admissible pair $(a,b)$ with $b<\infty$ and $\gamma_{a,b}=0$. Since $H^\beta$-solution is in particular an $H^\gamma$-solution, the uniqueness implies that $T \leq T^*$. We will show that $T$ is actually equal to $T^*$. Suppose that $T<T^*$, then the blowup alternative implies
\begin{align}
\|u(t)\|_{H^\beta} \rightarrow \infty \text{ as } t\rightarrow T. \label{regularity 1}
\end{align}
Moreover, since $T<T^*$, we have
\[
\|u\|_{L^p((0,T),H^\gamma_q)} + \sup_{0\leq t\leq T} \|u(t)\|_{H^\gamma} <\infty,
\]
where $(p,q)$ given in $(\ref{define pq})$. Using Strichartz estimate $(\ref{strichartz estimate})$, we have for any interval $I \subset (0,T)$,
\begin{align*}
\|u\|_{L^\infty(I,L^2)}+\|u\|_{L^p(I,L^q)} &\lesssim \|u_0\|_{L^2}+ \|F(u)\|_{L^{p'}(I,L^{q'})}, \\
\|u\|_{L^\infty(I,\dot{H}^\beta)} +\|u\|_{L^p(I,\dot{H}^\beta_q)} &\lesssim \|u_0\|_{\dot{H}^\beta} + \|F(u)\|_{L^{p'}(I,\dot{H}^\beta_{q'})}.
\end{align*}
Now, let $(m,n)$ be as in $(\ref{define mn})$. Lemma $\ref{lem nonlinear estimates}$, $(\ref{define mn})$ and Sobolev embedding then give
\begin{align*}
\|F(u)\|_{L^{p'}(I,L^{q'})} &\lesssim \|u\|^{\nu-1}_{L^p(I,L^n)} \|u\|_{L^m(I,L^q)} \lesssim \|u\|^{\nu-1}_{L^p(I,\dot{H}^\gamma_q)} \|u\|_{L^m(I,L^q)} \lesssim \|u\|_{L^m(I,L^q)}, \\
\|F(u)\|_{L^{p'}(I,\dot{H}^\beta_{q'})} &\lesssim \|u\|^{\nu-1}_{L^p(I,L^n)} \|u\|_{L^m(I,\dot{H}^\beta_q)} \lesssim \|u\|^{\nu-1}_{L^p(I,\dot{H}^\gamma_q)} \|u\|_{L^m(I,\dot{H}^\beta_q)} \lesssim \|u\|_{L^m(I,\dot{H}^\beta_q)}.
\end{align*}
Here we use the fact that $\|u\|_{L^p((0,T),H^\gamma_q)}$ is bounded. This shows that
\[
\|u\|_{L^\infty(I,H^\beta)} + \|u\|_{L^p(I,H^\beta_q)} \lesssim \|u_0\|_{H^\beta}+ \|u\|_{L^m(I,H^\beta_q)},
\]
for every interval $I \subset (0,T)$. Now let $0<\ep<T$ and consider $I=(0,\tau)$ with $\ep<\tau<T$. We have 
\[
\|u\|_{L^m(I,H^\beta_q)} \leq \|u\|_{L^m((0,\tau-\ep),H^\beta_q)} +\|u\|_{L^m((\tau-\ep,\tau),H^\beta_q)} \leq C_\ep + \ep^\theta \|u\|_{L^p(I,H^\beta_q)},
\]
where $\theta$ given in $(\ref{subcritical})$. Here we also use the fact that $u \in L^p_{\text{loc}}([0,T),H^\beta_q)$ since $\gamma_{p,q}=0$. Thus,
\[
\|u\|_{L^\infty(I,H^\beta)} + \|u\|_{L^p(I,H^\beta_q)} \leq C + C_\ep + \ep^\theta C \|u\|_{L^p(I,H^\beta_q)},
\]
where the various constants are independent of $\tau<T$. By choosing $\ep$ small enough, we have
\[
\|u\|_{L^\infty(I,H^\beta)}+\|u\|_{L^p(I, H^\beta_q)} \leq C,
\]
where $C$ is independent of $\tau<T$. Let $\tau \rightarrow T$, we get a contradiction with $(\ref{regularity 1})$. 
\subsection{The case $\gamma=d/2$}
Since $u_0 \in H^{d/2}$, Theorem $\ref{theorem d/2}$ shows that there exists a unique, maximal solution to (NL4S) satisfying $u \in C([0,T^*), H^{d/2}) \cap L^p_{\text{loc}}([0,T^*),L^\infty)$ for some $p>\max(\nu-1,4)$ when $d=1$ and $p>\max(\nu-1,2)$ when $d\geq 2$. This implies in particular that
\begin{align}
u \in L^{\nu-1}_{\text{loc}}([0,T^*), L^\infty). \label{regularity 2}
\end{align}
Now let $\beta>\gamma$. If $u_0 \in H^\beta$, then we know that $u$ is an $H^\beta$ solution defined on some maximal interval $[0,T)$ with $T\leq T^*$. Suppose that $T<T^*$. Then the unitary property of $e^{it\Delta^2}$ and Lemma $\ref{lem nonlinear estimates}$ imply that
\[
\|u(t)\|_{H^\beta} \leq \|u_0\|_{H^\beta} + \int_0^t \|F(u)(s)\|_{H^\beta}ds \leq \|u_0\|_{H^\beta}+ C\int_0^t \|u(s)\|^{\nu-1}_{L^{\infty}} \|u(s)\|_{H^\beta}ds,
\]
for all $0\leq t<T$. The Gronwall's inequality then yields
\[
\|u(t)\|_{H^\beta} \leq \|u_0\|_{H^\beta} \exp \Big( C\int_0^t \|u(s)\|^{\nu-1}_{L^\infty} ds\Big)
\]
for all $0\leq t<T$. Using $(\ref{regularity 2})$, we see that $\limsup \|u(t)\|_{H^\beta}<\infty$ as $t\rightarrow T$. This is a contradiction with the blowup alternative in $H^\beta$.
\subsection{The case $\gamma>d/2$}
Let $\beta >\gamma$. If $u_0 \in H^\beta$, then Theorem $\ref{theorem above d/2}$ shows that there is a unique maximal solution $u \in C([0,T), H^\beta)$ to (NL4S). By the uniqueness, we have $T\leq T^*$. Suppose $T<T^*$. Then 
\[
\sup_{0\leq t\leq T} \|u(t)\|_{H^\beta} <\infty,
\]
and hence
\[
\sup_{0\leq t \leq T} \|u(t)\|_{L^\infty} <\infty.
\]
This is a contradiction with the fact that $\limsup \|u(t)\|_{L^\infty}=\infty$ as $t\rightarrow T$. The proof of Theorem $\ref{theorem regularity}$ is now complete. \defendproof
\section{Ill-posedness} \label{section ill-posedness}
\setcounter{equation}{0}
In this section, we will give the proof of Theorem $\ref{theorem ill-posedness}$ using the technique of \cite{ChristCollianderTao}. We follow closely the argument of \cite{Dinhhalfwave}. Let us start with the small dispersion analysis.
\subsection{Small dispersion analysis}
Let us consider for $0<\delta \ll 1$ the following equation
\begin{align}
\left\{
\begin{array}{rl}
i\partial_t \phi(t,x) + \delta^4\Delta^2 \phi(t,x)&=-\mu |\phi|^{\nu-1} \phi(t,x), \quad (t, x) \in \R \times \R^d, \\
\phi(0,x) &= \phi_0(x), \quad x\in \R^d.
\end{array}
\right.
\label{small dispersion}
\end{align}
Note that $(\ref{small dispersion})$ can be transformed back to (NL4S) by using
\begin{align}
u(t,x):= \phi(t,\delta x). \label{back transformation}
\end{align}
\begin{lem} \label{lem small dispersion analysis}
Let $k>d/2$ be an integer. If $\nu$ is not an odd integer, then we assume also the additional regularity condition $\nu \geq k+1$. Let $\phi_0$ be a Schwartz function. Then there exists $C, c>0$ such that if $0<\delta\leq c$ sufficiently small, then there exists a unique solution $\phi^{(\delta)} \in C([-T,T],H^k)$ of $(\ref{small dispersion})$ with $T=c|\log \delta|^c$ satisfying
\begin{align}
\|\phi^{(\delta)}(t)-\phi^{(0)}(t)\|_{H^k} \leq C\delta^{3}, \label{small dispersion estimate}
\end{align}
for all $|t|\leq c|\log \delta|^c$, where 
\[
\phi^{(0)}(t,x):= \phi_0(x) \exp(-i\mu t|\phi_0(x)|^{\nu-1})
\]
is the solution of $(\ref{small dispersion})$ with $\delta=0$.
\end{lem}
\begin{proof}
We refer to Lemma 2.1 of \cite{ChristCollianderTao}, where the small dispersion analysis is invented to prove the ill-posedness for the nonlinear Schr\"odinger equation. The proof of Lemma $\ref{lem small dispersion analysis}$ is essentially given in Lemma 4.1 of \cite{Pausadercubic}, where the author treated the cubic fourth-order Schr\"odinger equation. The extension to the general power-type nonlinearity here is completely similar. Note that $H^k$ with $k>d/2$ is an algebra. 
\end{proof}
\begin{rem} \label{rem small dispersion analysis}
By the same argument as in \cite{ChristCollianderTao}, we can get the following better estimate
\begin{align}
\|\phi^{(\delta)}(t)-\phi^{(0)}(t)\|_{H^{k,k}} \leq C \delta^3, \label{small dispersion estimate weighted}
\end{align}
for all $|t| \leq c|\log \delta|^c$, where $H^{k,k}$ is the weighted Sobolev space
\[
\|\phi\|_{H^{k,k}}:= \sum_{|\alpha|=0}^{k} \|\scal{x}^{k-|\alpha|} D^\alpha \phi\|_{L^2}. 
\]
\end{rem}
Now let $\lambda>0$ and set
\begin{align}
u^{(\delta, \lambda)}(t,x):= \lambda^{-\frac{4}{\nu-1}} \phi^{(\delta)}(\lambda^{-4}t, \lambda^{-1} \delta x). \label{define solution}
\end{align}
It is easy to see that $u^{(\delta,\lambda)}$ is a solution of (NL4S) with initial data $u^{(\delta,\lambda)}(0)= \lambda^{-\frac{4}{\nu-1}} \phi_0(\lambda^{-1}\delta x)$. We have the following estimate of the initial data $u^{(\delta,\lambda)}(0)$.
\begin{lem} \label{lem initial data estimate}
Let $\gamma \in \R$ and $0<\lambda \leq \delta \ll 1$. Let $\phi_0 \in \Sh$ be such that if $\gamma \leq -d/2$,
\[
\hat{\phi}_0(\xi) = O(|\xi|^\kappa) \text{ as } \xi \rightarrow 0,
\]
for some $\kappa >-\gamma-d/2$, where $\hat{\cdot}$ is the Fourier transform. Then there exists $C>0$ such that 
\begin{align}
\|u^{(\delta,\lambda)}(0)\|_{H^\gamma} \leq C \lambda^{\Gace-\gamma} \delta^{\gamma-d/2}. \label{initial data estimate}
\end{align}
\end{lem}
The proof of this result follows the same lines as in Section 4 of \cite{ChristCollianderTao} for the nonlinear Schr\"odinger equation. We also refer to Lemma 3.3 of \cite{Dinhhalfwave} for the nonlinear half-wave context.
\subsection{Proof of Theorem $\ref{theorem ill-posedness}$}
We now give the proof of Theorem $\ref{theorem ill-posedness}$. We only consider the case $t\geq 0$, the one for $t<0$ is similar. Let $\ep \in (0,1]$ be fixed and set
\begin{align}
\lambda^{\Gact-\gamma} \delta^{\gamma-d/2} =:\ep, \label{define epsilon ill-posedness}
\end{align}
equivalently
\[
\lambda=\delta^\theta, \text{ where } \theta=\frac{d/2-\gamma}{\Gact-\gamma}>1,
\]
hence  $0 < \lambda \leq \delta \ll 1$. We note that we are considering the super-critical range, i.e. $\gamma <\Gact$. We will split the proof into several cases.
\paragraph{The case $0<\gamma<\Gact$.} We firstly have from Lemma $\ref{lem initial data estimate}$ and $(\ref{define epsilon ill-posedness})$ that 
\[
\|u^{(\delta,\lambda)}(0)\|_{H^\gamma} \leq C\ep.
\]
Since the support of $\phi^{(0)}(t,x)$ is independent of $t$, we see that for $t$ large enough, depending on $\gamma$,
\[
\|\phi^{(0)}(t)\|_{H^\gamma} \sim t^\gamma,
\]
whenever $\gamma \geq 0$ provided either $\nu>1$ is an odd integer or $\gamma \leq \nu-1$ otherwise. Thus for $\delta \ll 1$ and $1\ll t \leq c|\log \delta|^c$, $(\ref{small dispersion estimate})$ implies
\begin{align}
\|\phi^{(\delta)}(t)\|_{H^\gamma} \sim t^\gamma. \label{solution estimate}
\end{align}
Next, using  
\[
[u^{(\delta,\lambda)}(\lambda^4 t)]\hat{\ }(\xi)= \lambda^{-\frac{4}{\nu-1}} (\lambda\delta^{-1})^d [\phi^{(\delta)}(t)]\hat{\ } (\lambda\delta^{-1}\xi),
\]
we have
\begin{align*}
\|u^{(\delta,\lambda)}(\lambda^4 t)\|^2_{H^\gamma} &= \int (1+|\xi|^2)^\gamma |[u^{(\delta,\lambda)}(\lambda^4 t)]\hat{\ } (\xi)|^2 d\xi \\
&= \lambda^{-\frac{8}{\nu-1}} (\lambda\delta^{-1})^d \int (1+|\lambda^{-1}\delta\xi|^2)^\gamma |[\phi^{(\delta)}(t)]\hat{\ } (\xi)|^2 d\xi \\
&\geq \lambda^{-\frac{8}{\nu-1}} (\lambda\delta^{-1})^{d-2\gamma} \int_{|\xi|\geq 1} |\xi|^{2\gamma} |[\phi^{(\delta)}(t)]\hat{\ }(\xi)|^2 d\xi \\
&\geq \lambda^{-\frac{8}{\nu-1}} (\lambda\delta^{-1})^{d-2\gamma} \Big(c\|\phi^{(\delta)}(t)\|^2_{H^\gamma} - C\|\phi^{(\delta)}(t)\|^2_{L^2} \Big).
\end{align*}
We also have from $(\ref{solution estimate})$ that $\|\phi^{(\delta)}(t)\|_{L^2} \ll \|\phi^{(\delta)}(t)\|_{H^\gamma}$ for $t \gg 1$. This yields that
\[
\|u^{(\delta,\lambda)}(\lambda^4t)\|_{H^\gamma} \geq c \lambda^{-\frac{4}{\nu-1}} (\lambda\delta^{-1})^{d/2-\gamma} \|\phi^{(\delta)}(t)\|_{H^\gamma} \geq c \ep t^\gamma,
\]
for $1\ll t \leq c|\log \delta|^c$. We now choose $t=c|\log \delta|^c$ and pick $\delta>0$ small enough so that
\[
\ep t^\gamma > \ep^{-1}, \quad \lambda^4 t < \ep.
\]
Therefore, for any $\varepsilon>0$, there exists a solution of (NL4S) satisfying 
\[
\|u(0)\|_{H^\gamma} <\varepsilon, \quad \|u(t)\|_{H^\gamma}>\varepsilon^{-1}
\] 
for some $t \in (0,\varepsilon)$. Thus for any $t>0$, the solution map $\Sh \ni u(0) \mapsto u(t)$ for the Cauchy problem (NL4S) fails to be continuous at 0 in the $H^\gamma$-topology.

\paragraph{The case $\gamma \leq -d/2$ and $\gamma<\Gact$.} 
Using again Lemma \ref{lem initial data estimate} and $(\ref{define epsilon ill-posedness})$, we have 
\[
\|u^{(\delta,\lambda)}(0)\|_{H^\gamma} \leq C\ep,
\]
provided $0<\lambda\leq \delta \ll 1$ and $\phi_0 \in \Sh$ satisfying
\[
\hat{\phi}_0(\xi) = O(|\xi|^\kappa) \text{ as } \xi \rightarrow 0,
\]
for some $\kappa >-\gamma-d/2$. We recall that
\[
\phi^{(0)}(t,x)=\phi_0(x) \exp(-i\mu t|\phi_0(x)|^{\nu-1}).
\]
It is clear that we can choose $\phi_0$ so that
\[
\Big|\int \phi^{(0)}(1,x) dx \Big| \geq c \text{ or } |[\phi^{(0)}(1)]\hat{\ }(0)| \geq c,
\]
for some constant $c>0$. Since $\phi^{(0)}(1)$ is rapidly decreasing, the continuity implies that
\[
|[\phi^{(0)}(1)]\hat{\ } (\xi)| \geq c,
\]
for $|\xi| \leq c$ with $0<c \ll 1$. Since $H^{k,k}$ controls $L^1$ when $k>d/2$, $(\ref{small dispersion estimate weighted})$ implies
\[
|[\phi^{(\delta)}(1)]\hat{\ } (\xi) - [\phi^{(0)}(1)]\hat{\ } (\xi)| \leq C \delta^3,
\]
and then
\begin{align}
|[\phi^{(\delta)}(1)]\hat{\ }(\xi)| \geq c, \label{continuity}
\end{align}
for $|\xi|\leq c$ provided $\delta$ is taken small enough. We now have from $(\ref{define solution})$ that
\[
u^{(\delta,\lambda)}(\lambda^4,x)= \lambda^{-\frac{4}{\nu-1}} \phi^{(\delta)}(1,\lambda^{-1}\delta x),
\] 
and 
\[
[u^{(\delta,\lambda)}(\lambda^4)]\hat{\ }(\xi) = \lambda^{-\frac{4}{\nu-1}} (\lambda\delta^{-1})^d[\phi^{(\delta)}(1)]\hat{\ }(\lambda\delta^{-1}\xi).
\]
The estimate $(\ref{continuity})$ then yields
\[
[u^{(\delta,\lambda)}(\lambda^4)]\hat{\ }(\xi) \geq c \lambda^{-\frac{4}{\nu-1}} (\lambda\delta^{-1})^d,
\]
for $|\xi|\leq c \lambda^{-1}\delta$. \newline
\indent In the case $\gamma<-d/2$, we have from $(\ref{define epsilon ill-posedness})$ that
\[
\|u^{(\delta,\lambda)}(\lambda^4)\|_{H^\gamma} \geq c \lambda^{-\frac{4}{\nu-1}} (\lambda\delta^{-1})^d = c \ep (\lambda \delta^{-1})^{\gamma+d/2}.
\]
Here $0<\lambda\leq \delta \ll 1$, thus $(\lambda\delta^{-1})^{\gamma+d/2}\rightarrow +\infty$. We can choose $\delta$ small enough so that $\lambda \rightarrow 0$ and $(\lambda\delta^{-1})^{\gamma+d/2}\geq \ep^{-2}$ or
\[
\|u^{(\delta,\lambda)}(\lambda^4)\|_{H^\gamma} \geq \ep^{-1}.
\]
\indent In the case $\gamma=-d/2$, we have
\begin{align*}
\|u^{(\delta,\lambda)}(\lambda^4)\|_{H^{-d/2}} &\geq c \lambda^{-\frac{4}{\nu-1}} (\lambda\delta^{-1})^d \Big(\int_{|\xi|\leq c \lambda^{-1} \delta} (1+|\xi|)^{-d} d\xi\Big)^{1/2} \\
&=c \lambda^{-\frac{4}{\nu-1}} (\lambda\delta^{-1})^d (\log (c\lambda^{-1}\delta))^{1/2} \\
&=c \ep (\log (c\lambda^{-1}\delta))^{1/2}.
\end{align*}
By choosing $\delta$ small enough so that $\lambda \rightarrow 0$ and $\log (c\lambda^{-1}\delta) \geq \ep^{-4}$, we see that
\[
\|u^{(\delta,\lambda)}(\lambda^4)\|_{H^{-d/2}} \geq \ep^{-1}.
\]
Combining both cases, we see that the solution map fails to be continuous at 0 in $H^\gamma$-topology.
\paragraph{The case $\gamma=0 <\Gact$.} Let $a, a' \in [1/2,2]$. Let $\phi^{(a,\delta)}$ be the solution to $(\ref{small dispersion})$ with initial data
\[
\phi^{(a,\delta)}(0)=a \phi_0.
\]
Then, Lemma $\ref{lem small dispersion analysis}$ gives
\begin{align}
\|\phi^{(a,\delta)}(t)-\phi^{(a,0)}(t)\|_{H^k} \leq C \delta^3, \label{small dispersion estimate phi a delta}
\end{align}
for all $|t|\leq c|\log \delta|^c$, where 
\begin{align}
\phi^{(a,0)}(t,x)= a\phi_0(x) \exp (-i\mu a^{\nu-1} t|\phi_0(x)|^{\nu-1}) \label{define phi a zero}
\end{align} 
is the solution of $(\ref{small dispersion})$ with $\delta=0$ and the same initial data as $\phi^{(a,\delta)}$. Note that the constant $C, c$ above can be taken to be independent of $a$ since $a$ belongs to a compact set. We next define
\begin{align}
u^{(a,\delta,\lambda)}(t,x):= \lambda^{-\frac{4}{\nu-1}} \phi^{(a,\delta)}(\lambda^{-4}t, \lambda^{-1}\delta x). \label{define u a delta lambda}
\end{align}
Thanks to $(\ref{back transformation})$ and the scaling $(\ref{scaling})$, we see that $u^{(a,\delta, \lambda)}$ is also a solution of (NL4S). 
On the other hand, using $(\ref{define phi a zero})$, a direct computation shows that
\[
\|\phi^{(a,0)}(t)-\phi^{(a',0)}(t)\|_{L^2} \geq c >0,
\]
for some time $t$ satisfying $|a-a'|^{-1}\leq t\leq c|\log \delta|^c$ provided that $\delta$ is small enough so that $c|\log \delta|^c \geq |a-a'|^{-1}$. This estimate and $(\ref{small dispersion estimate phi a delta})$ yield
\[
\|\phi^{(a,\delta)}(t)-\phi^{(a',\delta)}(t)\|_{L^2} \geq c,
\]
for all $|a-a'|^{-1} \leq t \leq c|\log \delta|^c$. Now, let $\ep$ be as in $(\ref{define epsilon ill-posedness})$, i.e.
\begin{align}
\lambda^{-\frac{4}{\nu-1}}(\lambda\delta^{-1})^{d/2} =:\ep, \label{the choice of epsilon}
\end{align}
or $\lambda=\delta^\theta$ with $\theta=\frac{d/2}{\Gact}>1$. Moreover, using the fact 
\[
[u^{(a,\delta,\lambda)}(\lambda^4 t)]\hat{\ } (\xi) = \lambda^{-\frac{4}{\nu-1}} (\lambda\delta^{-1})^{d}[\phi^{(a,\delta)}(t)]\hat{\ }(\lambda\delta^{-1}\xi),
\]
we obtain
\[
\|u^{(a,\delta,\lambda)}(\lambda^4t)-u^{(a',\delta,\lambda)}(\lambda^4 t)\|_{L^2} = \lambda^{-\frac{4}{\nu-1}}(\lambda\delta^{-1})^{d/2} \|\phi^{(a,\delta)}(t)-\phi^{(a',\delta)}(t)\|_{L^2} \geq c \ep. 
\]
Similarly, using 
\[
[u^{(a,\delta,\lambda)}(0)]\hat{\ } (\xi)= a \lambda^{-\frac{4}{\nu-1}} (\lambda\delta^{-1})^d \hat{\phi}_0(\lambda \delta^{-1}\xi),
\]
the choice of $\ep$ in $(\ref{the choice of epsilon})$ gives
\[
\|u^{(a,\delta,\lambda)}(0)\|_{L^2}, \|u^{(a',\delta,\lambda)}(0)\|_{L^2} \leq C \ep,
\]
and
\[
\|u^{(a,\delta,\lambda)}(0)-u^{(a',\delta,\lambda)}(0)\|_{L^2} \leq C\ep|a-a'|.
\]
Since $|a-a'|$ can be arbitrarily small, this shows that for any $0<\ep, \sigma <1$ and for any $t>0$, there exist $u_1, u_2$ solutions of (NL4S) with initial data $u_1(0), u_2(0) \in \Sh$ such that
\[
\|u_1(0)\|_{L^2}, \|u_2(0)\|_{L^2} \leq C\ep, \quad \|u_1(0)-u_2(0)\|_{L^2} \leq C\sigma, \quad \|u_1(t)-u_2(t)\|_{L^2} \geq c\ep.
\]
This shows that the solution map fails to be uniformly continuous on $L^2$. This completes the proof of Thereom $\ref{theorem ill-posedness}$.
\section*{Acknowledgments}
The author would like to express his deep gratitude to Prof. Jean-Marc BOUCLET for the kind guidance, encouragement and careful reading of the manuscript. 
\addcontentsline{toc}{section}{Acknowledments}

{\sc Institut de Math\'ematiques de Toulouse, Universit\'e Toulouse III Paul Sabatier, 31062 Toulouse Cedex 9, France.} \\
\indent Email: \href{mailto:dinhvan.duong@math.univ-toulouse.fr}{dinhvan.duong@math.univ-toulouse.fr}
\end{document}